\documentclass[a4paper,10pt]{article}
\usepackage[utf8]{inputenc}
\usepackage[T1]{fontenc}
\usepackage[english]{babel}
\usepackage{amsmath}
\usepackage{amsthm}
\usepackage{amsfonts}
\usepackage{amssymb}
\usepackage{listings}
\usepackage{graphicx}
\usepackage{mathrsfs}
\usepackage{hyperref}
\usepackage{algpseudocode}
\usepackage{algorithmicx}
\usepackage{algorithm}
\usepackage{subfig}
\usepackage{stmaryrd}
\usepackage{bm}
\usepackage{tikz}
\usepackage{braket}
\usepackage{wasysym}
\usepackage{wrapfig}
\usepackage[framemethod=TikZ]{mdframed}
\usepackage{xcolor}
\usepackage{pstricks}
\usepackage{faktor}

\usepackage[top=2.00cm,bottom=3.00cm,left=2.00cm,right=2.00cm]{geometry}

\newcommand{\mc}{\mathscr}
\newcommand{\f}{\mathbb}
\newcommand{\C}{\mathbb{C}}

\newcommand{\ol}{\overline}
\newcommand{\cu}{\subseteq}
\newcommand{\wt}{\widetilde}

\newcommand{\HC}{\mathcal H(S^1)}
\newcommand{\PC}{\mathcal P(S^1)}
\newcommand{\HCM}[1]{\HC^{#1\times #1}}

\DeclareMathOperator{\tr}{Tr}

\newtheorem{theorem}{Theorem}[section]
\newtheorem{definition}[theorem]{Definition}
\theoremstyle{plain}

\newtheorem{example}[theorem]{Example}
\newtheorem{question}[theorem]{Question}
\newtheorem{lemma}[theorem]{Lemma}
\newtheorem{remark}[theorem]{Remark}
\newtheorem{proposition}[theorem]{Proposition}
\newtheorem{corollary}[theorem]{Corollary}

\title{On the Rellich eigendecomposition of para-Hermitian matrices and the sign characteristics of $*$-palindromic matrix polynomials}
\author{Giovanni Barbarino\footnote{Corresponding author. Department of Mathematics and Systems Analysis, Aalto University, PO Box 11100, 00076 Aalto,
Finland. Supported by the Alfred Kordelinin s\"{a}\"{a}ti\"{o} Grant No. 210122.} \, and Vanni Noferini\footnote{Department of Mathematics and Systems Analysis, Aalto University, PO Box 11100, 00076 Aalto,
Finland. Supported by an Academy of Finland grant (Suomen Akatemian p\"{a}\"{a}tos 331240).}}
\date{}

\begin{document}

\maketitle

\begin{abstract}
We study the eigendecompositions of para-Hermitian matrices $H(z)$, that is, matrix-valued functions that are analytic and Hermitian on the unit circle $S^1 \subset \C$. In particular, we fill existing gaps in the literature and prove the existence of a decomposition $H(z)=U(z)D(z)U(z)^P$ where, for all $z \in S^1$, $U(z)$ is unitary, $U(z)^P=U(z)^*$ is its conjugate transpose, and $D(z)$ is real diagonal; moreover, $U(z)$ and $D(z)$ are analytic functions of $w=z^{1/N}$ for some positive integer $N$, and $U(z)^P$ is the so-called para-Hermitian conjugate of $U(z)$. This generalizes the celebrated theorem of Rellich for matrix-valued functions that are analytic and Hermitian on the real line. We also show that there also exists a decomposition $H(z)=V(z)C(z)V(z)^P$ where $C(z)$ is pseudo-circulant,  $V(z)$ is unitary and both are analytic in $z$. We argue that, in fact, a version of Rellich's theorem can be stated for matrix-valued function that are analytic and Hermitian on any line or any circle on the complex plane. Moreover, we extend these results to para-Hermitian matrices whose entries are Puiseux series (that is, on the unit circle they are analytic in $w$ but possibly not in $z$). Finally, we discuss the implications of our results on the singular value decomposition of a matrix whose entries are $S^1$-analytic functions of $w$, and on the sign characteristics associated with unimodular eigenvalues of $*$-palindromic matrix polynomials.
\end{abstract}

\textbf{Keywords:} para-Hermitian, analytic eigendecomposition, palindromic matrix polynomial, para-unitary, sign characteristic, Rellich's theorem

\textbf{MSC:} 15A23, 15A18, 15A54, 15B57

\section{Introduction}

In the 1930s, F. Rellich proved a famous theorem \cite{rellichorig} on the existence of unitary real-analytic eigendecompositions of real-analytic matrix-valued functions that are Hermitian on the real line. Rellich's result, that has since become classical \cite{glr,kato,rellich}, is stated below in a slightly more general version that only assumes real-analyticity on an interval (see e.g. \cite{adkins,grater,mntx,wimmer}).

\begin{theorem}[Rellich's Theorem]\label{thm:rellich}
Let each entry of the matrix $A(x)$ be a function analytic in the open interval $I \subseteq \mathbb{R}$, and suppose that that $A(x)=A(x)^*$ is a Hermitian matrix for every $x \in I$. Then there exists an analytic eigenvalue decomposition of the form 
\[
A(x) = U(x) D(x) U(x)^*
\]
where $U(x)$ is unitary for every $x \in I$ and $D(x)$ is real and diagonal for every $x\in I$; moreover, all the entries of $D(x)$ and $U(x)$ are analytic in $I$.
\end{theorem}
The main goal of this paper is to analyze how and when Rellich's theorem can be generalized to the case of a matrix that is analytic and Hermitian on the unit circle. While it is tempting to conjecture that such a generalization is a corollary of Theorem \ref{thm:rellich} via an appropriate change of variable such as, e.g., $x=e^{it}$, we will see that actually the different topology of the unit circle induces several subtleties and complications with respect to the case of the real line.

There are at least two motivations to study this problem and attempt to clarify the above mentioned subtleties. On one hand, in the context of signal processing several algorithms of practical importance rely on the numerical computation of an eigendecomposition of matrices that are analytic and Hermitian on the unit circle $S^1$ (called para-Hermitian in the signal processing literature): see \cite{EVD1} and the references therein for a survey of the numerous engineering applications. A related problem is to compute the singular value decomposition of matrices that are analytic on $S^1$. The signal processing literature has mainly focused on the development of numerical algorithms that could achieve the task of approximating these decompositions. Some papers did tackle the mathematical theory, but occasionally imprecise or incorrect statements have appeared. The article \cite{EVD2} stands out in terms of mathematical quality: indeed, it provided an almost complete proof of an analogue of Rellich's theorem on the unit circle. Unfortunately, though, the proof in \cite{EVD2} still had a gap concerning the existence of the (orthonormal) eigenvectors. The present paper aims to fill this gap and to provide a fully complete theoretical analysis. On the other hand, in the context of numerical linear algebra it is known that Rellich's theorem provides one way to define the sign characteristics of real eigenvalues of Hermitian matrix polynomials and Hermitian matrix functions: see \cite{glrind,glr,mntx} and the references therein. The sign characteristic is important also for the unimodular eigenvalues of *-palindromic matrix polynomials, another class of structured matrix polynomials of interest in numerical linear algebra \cite{m4}. In fact, a *-palindromic matrix polynomial is equal to a para-Hermitian matrix times a (not always analytic, depending on the parity of the grade of palindromicity) scalar function. Thus, the theory developed in this paper is also useful for the study of the sign characteristics of *-palindromic matrix polynomials.

The structure of the paper is as follows. In Section \ref{sec:bg}, we recall background material including previous generalizations of Rellich's Theorem, previous partial results on the case of matrix-valued functions Hermitian on the unit circle, and notions of algebra and analysis that are useful to our further developments. In Section \ref{sec:evd}, we prove our main results concerning the existence of unitary decompositions for certain classes of matrix-valued functions that are Hermitian on the unit circle, including analytic para-Hermitian matrices that have only analytic eigenvalues (they admit an analytic unitary eigendecomposition) and para-Hermitian matrices that are analytic functions of $w=z^{1/N}$ for some positive integer $N$ (they admit a unitary eigendecomposition in the same class); we also discuss how to further generalize Rellich's theorem to matrix-valued functions that are analytic and Hermitian on an arbitrary line or on an arbitrary circle in the complex plane, and we give a proof of the existence of the so-called pseudo-circulant eigendecomposition. In Section \ref{sec:svd}, we apply our results to discuss the singular value decomposition of any matrix, possibly rectangular, whose entries belong to the previously mentioned class of functions that are analytic in $w=z^{1/N}$ on the unit circle, for some $N$; this is an extension of a result proved in \cite{Weissnew} for matrices that are analytic on the unit circle $(N=1)$, and even for that case our proof is new. In Section \ref{sec:signchar} we turn our attention to the implications for the theory of the sign characteristics of unimodular eigenvalues of $*$-palindromic matrix polynomials. We finally draw some conclusions in Section \ref{sec:conclusions}.

\section{Background material}\label{sec:bg}
Recall that if $z=x+i y$ ($x,y \in \mathbb{R}$) its complex conjugate is ${\ol z}=x-\textnormal iy$. For a matrix $A=[a_{ij}] \in \C^{m \times n}$, we denote its conjugate transpose by $A^*=[{\ol a_{ji}}] \in \C^{n \times m}$.

Rellich's theorem was previously generalized in several papers, such as \cite{adkins,grater,wimmer}. There, the authors extended the result to different kind of functions on subsets of the complex plane, but always under the assumption that the matrix function $A(x)$ is Hermitian on a real interval. As a consequence, they defined the conjugation operation on square matrices as $A(z) = [a_{i,j}(z)]_{1 \leq i,j \leq n} \to A(z)^* = [\ol {a_{j,i}(\ol z)}]_{1 \leq i,j \leq n}$. The reason for this choice stems from the classic representation of scalar power series and its canonical conjugation operation as
\[
f(z) = a_0 + a_1 z +a_2 z^2 +\dots \quad \longrightarrow \quad f(z)^* = \ol{f(\ol z)} =  \ol a_0 + \ol a_1 z +\ol a_2 z^2 +\dots
\]
from which one can easily derive that $f(z)^* = \ol{f(\ol z)} = \ol {f(z)}$ when $z$ is a real number, but in general for no other point.  Moreover, in \cite{adkins,grater,wimmer} the authors always focused on functions defined on a domain $\Omega$ that is symmetric with respect to the real line: this implies that the conjugation $^*$ is an isomorphism of $\mathcal H(\Omega)$, the ring of holomorphic functions on $\Omega$.

In this paper, we are interested in $\HC$, defined as the space of function that are holomorphic on the unit circle $S^1:=\{z \in \C : |z|=1  \}$ (note in passing that, if a function $f(z) \in \HC$, then $f(z)$ also has an holomorphic extension on some open neighbourhood of $S^1$). As a consequence, we will work with matrices $A(z)\in \HCM{n}$ that are Hermitian for every $z\in S^1$. 
\begin{definition}\label{def:pHc}
Given $A(z)= [a_{i,j}(z)]_{1 \leq i\leq m,1\leq j \leq n} \in \HC^{m\times n}$, its para-Hermitian conjugate is defined as $A(z)^P = [\ol {a_{j,i}(\ol z^{-1})}]_{1 \leq i\leq m,1\leq j \leq n}$. Moreover a matrix $R(z) \in \HCM{n}$ is called a para-Hermitian matrix if $R(z)=R(z)^P$ and a matrix $U(z) \in \HCM{n}$ is called para-unitary if $U(z)U(z)^P=U(z)^PU(z)=I$.
\end{definition}
\begin{example}\label{ex1}
Consider
$$A(z) = \begin{bmatrix}
1+\textnormal i & z\\
\exp(z) & 0
\end{bmatrix}, \qquad R(z)=\begin{bmatrix}
0 & 1+z^{-1}\\
1+z & 0
\end{bmatrix}, \qquad U(z)=\frac{1}{\sqrt{2}} \begin{bmatrix}
z&z\\
z&-z
\end{bmatrix}.$$
Then $R(z)$ is para-Hermitian, $U(z)$ is para-unitary, and $A(z)^P=\begin{bmatrix}
1-\textnormal i & \exp(z^{-1})\\
z^{-1} & 0
\end{bmatrix}.$
\end{example}

Note that for a scalar function $f(z)^P:=\ol{f(\ol z^{-1})} = \ol {f(z)}$ when $z\in S^1$, but in general for no other point. Para-Hermitian conjugation as defined in Definition \ref{def:pHc} is indeed an isomorphism of the ring $\HC$ and for any $z_0 \in S^1$, $A(z_0)^P = A(z_0)^*$. Observing that para-Hermitian (resp. para-unitary) matrices are Hermitian (resp. unitary) on $S^1$, it is a natural question whether an analogue of Rellich's Theorem \ref{thm:rellich} holds for para-Hermitian matrices.
\begin{question}\label{q1}
	Given a para-Hermitian matrix $A(z)\in \HCM{n}$, when does it admit an analytic unitary eigendecomposition? That is, do there always exist $U(z),D(z)\in \HCM n$ such that $U(z)$ is para-unitary, $D(z)$ is diagonal real for all $z \in S^1$, and $A(z) = U(z)D(z)U(z)^P$?
\end{question}
\noindent The authors of \cite{EVD2} answered Question \ref{q1} negatively, by providing the counterexample $R(z)+2I_2$, where $R(z)$ is the para-Hermitian matrix of Example \ref{ex1}\footnote{Adding $2 I_2$ to $R(z)$ makes the para-Hermitian matrix also positive semidefinite, in the sense that its eigenvalue functions become nonnegative on $S^1$; in signal processing, an important role is played by the so-called cross spectral density matrices, which are both para-Hermitian and positive semidefinite. Having this in mind, the authors of \cite{EVD2} wanted to construct a physically relevant counterexample. Here we are mainly focused on the mathematical aspects of the discussion, and positive semidefiniteness is not relevant to our goals.}.

To see why $R(z)$ (and hence $R(z)+2I_2$) does not have $S^1$-analytic eigenvalues, note that $\tr R(z)=0,\det R(z)=-(z+z^{-1}+2)$, and hence the eigenvalues of $R(z)$ are $\lambda_{\pm}=\pm(z^{1/2}+z^{-1/2})$. This example exposes a striking difference between Hermitian analytic matrix functions (whose eigenvalues are always all real-analytic) and para-Hermitian analytic matrices (that may have some eigenvalues that are not analytic on the whole $S^1$). In fact, in \cite{EVD2} it is stated that any para-Hermitian matrix $R(z) \in \HCM n$ has a unitary eigendecomposition in the space of 
convergent Laurent series with respect to the variable $z^{1/N}$ for some positive integer $N$; moreover, it is claimed that if the eigenvalues happen to be holomorphic then a unitary decomposition actually exists over $\HC$. While the claims in \cite{EVD2} are correct, the proof therein is not fully complete. Indeed, the authors only directly prove the existence of the eigenvalues; for the eigenvectors, they attribute a proof of their existence to Rellich, but in fact this is not an immediate corollary of neither Theorem \ref{thm:rellich} nor its generalizations that have appeared in the literature. %The reasoning used in \cite{EVD2} to prove the existence of the eigenvalues does not generalizes to the eigenvectors. 
To fill this subtle gap, a new argument is thus necessary to properly complete the proof of the existence of the various eigendecompositions. Below, we will complete the proofs in \cite{EVD2} and fill in all the missing justifications and details.

%% NOTICE: We cannot define the PH as the classic conjugation because \ol z is NOT Holomorphic, anywhere.
%% The fact that we need the change of variables z -> z/|z|^2 tells us that the domain should be invariant under this transformation, so it must be a subset of S^1... 

\subsection{Elementary divisor domains, holomorphic functions and Laurent series}\label{sec:Laurent}
In this subsection, we recall some basic notions in algebra \cite{EDD} and in analysis \cite{Rudin} that are needed in the paper.

\noindent Recall that in ring theory an integral domain $R$ is a ring (with unity) that does not contain nonzero divisors of zero. A square matrix $A \in R^{n \times n}$ is called non-singular, or regular, if $\det A \neq 0$, or equivalently if there is no nonzero $v \in R^n$ such that $Av=0$; and it is called unimodular if $\det A$ is a unit (invertible element) of $R$, or equivalently if there exists $A^{-1} \in R^{n \times n}$ such that $A^{-1}A=AA^{-1}=I$. The set of units of $R$ is a group with respect to multiplication, and it is denoted by $R^\times$.

Algebraically, since $S^1$ is a connected set, the space $\HC$ is an Elementary Divisors Domain (EDD) \cite[Theorem 1.5.3]{EDD}, and in particular an integral domain. Even more strongly, since $S^1$ is also compact, then $\HC$ is in fact a Principal Ideal Domain (PID) \cite[Corollary 1.2.7, Lemma 1.3.8]{EDD}. As a consequence, any matrix $A(z)\in \HC^{n\times m}$ admits a Smith canonical Form \cite[Theorem 1.14.1]{EDD}. We state this result below in Theorem \ref{thm:smith}. To this goal, recall that the rank of a matrix over any integral domain is the maximal size of a minor with non-vanishing determinant. Moreover, observe that the units, i.e., invertible elements, of $\HC$ are the analytic functions with no zeros on $S^1$, or more formally $\HC^\times :=\{ f(z) \in \HC \ \mathrm{s.t.} \ f(z_0) \ne 0 \ \forall \ z_0 \in S^1 \}$.
\begin{theorem}\label{thm:smith}
Let $A(z)\in \HC^{n\times m}$. Then, there exist unimodular matrices $M(z)\in \HCM{n}$ and $N(z)\in \HCM m$ such that $S(z)=M(z)A(z)N(z)$ is (possibly rectangular and) diagonal, with diagonal elements $s_1(A), s_2(A),\ldots, s_r(A),0,\dots,0$ where  $r$ is the rank of $A(z)$ and  $s_{j-1}(A) | s_j(A)$ for $j=2,\dots,r$. The matrix $S(z)$ is uniquely determined (up to multiplication of its nonzero diagonal elements by units in $\HC^\times$) by $A(z)$ and it is called the \emph{Smith canonical form} of $A(z)$.
\end{theorem}

 %, since by continuity there exists an open annulus containing $S^1$ where they are non-zero and holomorphic.
%A function $f(z)\in \HC$ is uniquely determined by its values on $S^1$. In fact, 
Changing our viewpoint from algebra to analysis, observe that $f(z) \in \HC$ if and only if it can be represented as a Laurent series $f(z)=\sum_{n\in\f Z} a_nz^n$ with radii of convergence $r<1<R$, that is, with coefficients $a_n$ that are exponentially decreasing on both sides. This in turn be equivalently rewritten as $F(\theta) := f(e^{\textnormal i \theta})$ being a $2\pi$-periodic analytic function on $\f R$. Conversely, given such a function $F(\theta)$ one can generate the original $f(z)$ substituting $z$ to $e^{\textnormal i\theta}$ in its Fourier series and automatically obtain $f(z)\in\HC$.

A Laurent series $\sum_{n\in \f Z} a_nz^n$ is said to be convergent when it is holomorphic on an open annulus containing $S^1$, that is $|a_n|^{1/n}\xrightarrow{n\to +\infty} R^{-1} < 1$ and $|a_n|^{1/n}\xrightarrow{n\to -\infty} r < 1$, where $r,R$ are the inner and outer radii of the annulus. Analogously, for any positive integer $N$, we can consider the series $\sum_{n\in \f Z} a_nz^{n/N}$ and say it is convergent if the associated Laurent series $\sum_{n\in \f Z} a_nz^{n}$ is convergent.
Here the fractional exponential is defined for any $z\ne 0$ as $z^{1/N} := |z|^{1/N} \exp(\textnormal i \arg(z)/N)$ with $\arg(z)\in ]-\pi,\pi]$. 
 We can thus define $\mathcal H_N(S^1)$ to be the space of convergent Laurent series with variable $z^{1/N}$ and $\mathcal P(S^1) := \bigcup_{N>0}   \mathcal H_N(S^1)$. Note that $\HC =\mathcal H_1(S^1)\cu \mathcal P(S^1)$ and that the (scalar) para-Hermitian conjugation of Definition \ref{def:pHc} $f(z)^P:=\ol{f(\ol z^{-1})}$ is still well-defined on $\mathcal P(S^1)$ and it still coincides with the classic complex conjugation on $S^1$, since  $\ol z^{1/N}= \ol {z^{1/N}}$ for all $z\ne 0$. 
 Moreover, $z\mapsto z^{1/N}$ is holomorphic on $\f C\setminus (\f R^- \cup \{0\})$ and $S^1$ is invariant. As a consequence, if $f(z)\in \HC$ then necessarily $f(z^{1/N}) \in\PC$ is still holomorphic on an open neighbourhood of $S^1$ minus the negative real semiline, but  
  in general $f(z^{1/N})$  is not even continuous on $S^1$. At the same time, consider the function $\wt F(\theta) = f(\exp(\textnormal i\theta)^{1/N})$ on the real interval $]-\pi,\pi]$ and observe that on this domain it coincides with $F(\theta) =  f(\exp(\textnormal i\theta/N))$, but $F(\theta)$ is analytic on $\f R$ and $2\pi N$-periodic. $\wt F(\theta)$ is thus analytic in $]-\pi,\pi[$ and continuous at $\pi$. Conversely, given such a function $F(\theta)$ one can generate the original $f(z^{1/N})$ substituting $z$ to $e^{\textnormal i\theta}$ in its Fourier series and automatically obtain $f(z)\in\mathcal H_N(S^1)$.

\section{Unitary eigendecompositions of some classes of matrix-valued functions that are Hermitian on the unit circle}\label{sec:evd}
\subsection{Existence of the holomorphic unitary eigendecomposition in the case of holomorphic eigenvalues}\label{sec:goodevals}

%Here 
%% we cannot extend much, just to subsets of S^1...

We start by analyzing a special situation that will also be useful to tackle the more general case. That is, we assume that all the eigenvalues of the para-Hermitian $A(z)\in\HCM n$ are in $\HC$. This assumption suffices to establish an $\HC$-analogue of Rellich's theorem. Our proof below mainly follows the lead of \cite{wimmer}, generalizing the results therein to the case of our interest.

\begin{lemma}\label{lem:normalization}
	Let $w(z)\in \HC^n$ be such that $w(z_0)\ne 0$ for all $z_0 \in S^1$. Then there exists $\gamma(z) \in \HC^\times$ such that for $v(z) = w(z)\gamma(z)$ we have $v(z)^P v(z) = 1$. 
\end{lemma}
\begin{proof}
	Since for any $z_0\in S^1$ the value $w(z_0)$ is never zero and  $w(z_0)^P = w(z_0)^*$, we have that $w(z_0)^Pw(z_0) = \|w(z_0)\|^2>0$. Using the characterization of $\HC$ in terms of Laurent series (see Subsection \ref{sec:Laurent}), and since $w(z)^Pw(z) \in \HC$, the function $\alpha(\theta)  = w(e^{\textnormal i\theta})^Pw(e^{\textnormal i\theta})$ is real, positive, $2\pi$-periodic and analytic on $\f R$. Moreover, the function $x\mapsto x^{-1/2}$ is analytic on $]0,\infty[$, so 
	$\beta(\theta) := \alpha(\theta)^{-1/2} = \|w(e^{\textnormal i\theta})\|^{-1}$ is still analytic, positive, and $2\pi$-periodic. It thus induces a function $\gamma(z)\in\HC^\times$ such that $\gamma(e^{\textnormal i\theta}) = \beta(\theta)= \|w(e^{\textnormal i\theta})\|^{-1}$. If now $v(z)= w(z)\gamma(z)$, then we have
	$v(z_0)^Pv(z_0) = \|v(z_0)\|^2 = 1$  for any $z_0\in S^1$, but $v(z)^Pv(z) \in \HC$, so it must necessarily be identically 1 on the domain of convergence of $v(z)$.
\end{proof}

\begin{lemma}\label{lem:GS}
	Given two column vectors $v(z),w(z) \in \HC^{n}$, define  the `para-Hermitian scalar product' $\langle v(z),w(z)\rangle :=  v(z)^P w(z)$. 
	Given a unimodular matrix $N(z) \in\HCM n$, the Gram-Schmidt algorithm applied to its columns with respect to the para-Hermitian scalar product $\langle \cdot,\cdot\rangle$ produces a para-unitary matrix.
\end{lemma}
\begin{proof}
	Observe that the unimodularity of the matrix $N(z)\in \HCM n$ is equivalent to $\det N(z)$ begin a unit of $\HC$, or in other words an analytic function without zeros on $S^1$. Hence, for any vector $v(z)$ which is a column of $N(z)$, necessarily we have $\det(N(z_0))\ne 0\implies v(z_0)\ne 0$ for all $z_0\in S^1$.
	
	We need to show that all the steps of the Gram-Schmidt method are actually well-defined inside the ring $\HC$ and preserve the unimodularity of the matrix $N(z)$ to which it is applied. More formally, we want to show that the Gram-Schmidt algorithm implicitly defines a (finite) sequence of unimodular matrices $N_k(z) \in \HC^{n \times n}$ where $N_0(z):=N(z)$. Observe that each step in the Gram-Schmidt sequence constructs $N_{k+1}(z)$ from $N_k(z)$ either by normalizing a column or by orthogonalizing a column with respect to another one.
	
	Let us first consider the normalization steps. From Lemma \ref{lem:normalization} we know that we can always normalize a vector in $\HC^{n}$ that is never zero on $S^1$ by multiplication times a function $\gamma(z)\in \HC^\times$. Suppose that, before such a normalization step, the matrix $N_k(z)$ in the sequence is unimodular, that is, $\det N_k(z) \in \HC^\times$; then, the normalization produces a new matrix $N_{k+1}(z)$ and $\det N_{k+1}(z)=\gamma(z) \det N_k(z) \in \HC^\times$. In other words, the normalization steps within the Gram-Schmidt algorithm preserve the unimodularity of the matrices in the sequence $N_k(z)$, and are always well-defined in $\HC$.
	
	It remains to discuss what happens during the orthogonalization steps. There, we take two columns $v(z),w(z)$ of the unimodular matrix $N_h(z)$ in the sequence, where $w(z)^Pw(z)=1$, and then replace $v(z)$ by the vector 
	$\wt v(z) = v(z) - w(z)\langle v(z),w(z)\rangle$. Since para-Hermitian conjugation is an isomorphism of $\HC$, and since $\HC$ is a ring, clearly $\wt v(z)\in \HC^n$. Moreover the step is equivalent to setting $N_{h+1}(z)=N_h(z) [I-e_i e_j^T \langle v(z),w(z) \rangle]$, where $i,j$ are indices such that $v(z)=N_h(z) e_j$ and $w(z)=N_h(z) e_i$. Hence, $\det N_{h+1}(z)=\det N_h(z)$ and the orthogonalization step also preserves unimodularity of the matrices in the sequence.
	
	Hence, the Gram-Schmidt algorithm can be applied to the columns of any unimodular matrix $N(z)\in\HCM n$, and it eventually construct as an output another unimodular matrix. Moreover, the output is a matrix $U(z)\in\HCM n$ with orthonormal columns, that, is $U(z)^PU(z) =I$. 
\end{proof}

\begin{theorem}\label{theo:analytic_eigvl_imply_analytic_eigvc}
	%Let $\f D$ be an EDD with a conjugation operator $P:\f D\to \f D$ such that 
	%Let $\f D = \mathcal H(K)$ for any $K\cu \f C$ connected set. 
	Given a para-Hermitian matrix $A(z) \in \HCM n$ and $n$ functions $\mu_1(z),\dots,\mu_n(z) \in \HCM n$ such that the eigenvalues of $A(z_0)$ are precisely  $(\mu_1(z_0),\dots,\mu_n(z_0)) \in \mathbb{R}^n$ for any $z_0\in S^1$, then there exist a para-unitary matrix $U(z)\in \HCM{n}$ such that $A(z) = U(z)D(z)U(z)^P$, where $D(z) \in \HCM n$ is diagonal and its $(i,i)$ element, for $i=1,\dots,n$, is equal to $\mu_i(z)$. 
\end{theorem}
\begin{proof}
	From the assumptions, we know that $\det(\lambda I - A(z_0)) = \prod_{i=1}^n (\lambda - \mu_i(z_0))$ for every $z_0\in S^1$. 
	Since $\mu_i(z_0)$ is an eigenvalue of the Hermitian matrix $A(z_0)$ for all $z_0\in S^1$, it holds that $\mu_i(S^1)\cu \f R$.  As a consequence, $\mu_i(z) =\mu_i(z)^P$, as  for any $z_0\in S^1$ they coincide since   $\mu_i(z_0)=\ol{\mu_i(z_0)}=\mu_i(z_0)^P$.  
	Let $L(z) := A(z) - \mu_1(z) I\in \HCM n$  and note that $\det(L(z_0))=0$ for all $z_0\in S^1$, but since $\det(L(z))\in \HC$ we have that $\det(L(z)) =0$.
	In addition, $L(z)$ is para-Hermitian because $L(z)^P=A(z)^P - \mu_1(z)^PI = L(z)$.  $\HC$ is a PID, and hence an EDD (see Section \ref{sec:Laurent}). Therefore  $L(z)$ admits a Smith Normal Form $L(z) = M(z)S(z)N(z)$ with $M(z),N(z)$ invertible and $S(z)$ diagonal and such that, for all $i$, its $i$-th diagonal element divides its $(i+1)$-th diagonal element. Since the rank of $L(z)$ is not full, by the divisibility property stated above the last diagonal element of $S(z)$ is surely $0$ and $S(z)e_n = 0\implies L(z) N(z)^{-1}e_n = 0$. Defining $w(z) := N(z)^{-1}e_n$, then $\det(N(z_0)^{-1})\ne 0\implies w(z_0) = 
	N(z_0)^{-1} e_n \ne 0$ for all $z_0\in S^1$.  From Lemma \ref{lem:normalization}, we can normalize $w(z)$ to obtain a null vector of unit norm $v(z)\in \HC^{n}$, i.e., $v(z)^Pv(z) = 1$  and $L(z) v(z) = 0$. %Considering $v(z)$ as a $n\times 1$ matrix, we can write its Smith form as $R(z)v(z)\alpha(z) = e_1\beta(z)$, where $\alpha(z)$ and $R(z)$ are invertible. Observe that $\det(R(z_0))$, $v(z_0)$ and $\alpha(z_0)$ cannot be zero for any $z_0\in S^1$, therefore $\beta(z)$ must have no zeros on $S^1$, making it an invertible function. One can thus rewrite the Smith form as $Y(z)^P v(z) =e_1$ with $Y(z) = [\beta(z)^{-1}\alpha(z)R(z)]^P$ being still an invertible matrix. 
Moreover, the equation $v(z)^P v(z)=1$ shows that $v(z)$ is left invertible, and thus it has a trivial Smith normal form \cite[Theorem 3.3]{zaballa}, i.e., there is an invertible (over $\HC$) matrix $Y(z)$ such that $Y(z)v(z)=e_1$.		Let $y_i(z)$ be the columns of $Y(z)$. By Lemma \ref{lem:GS}, we can orthonormalize the invertible matrix $\big[ y_2(z) \, y_3(z)\,\dots\,y_n(z)\, y_1(z)\big]$ into the para-unitary matrix $\big[ \wt y_2(z) \,\wt  y_3(z)\,\dots\,\wt y_n(z)\, \wt y_1(z)\big]$, but since $y_i(z)^Pv(z) = 0$ for all $i>1$ then it also holds that $\wt y_i(z)^Pv(z) = 0$ for all $i>1$ and as a consequence the matrix $V_1(z) = \big[v(z)\, \wt y_2(z) \,\wt  y_3(z)\,\dots\,\wt y_n(z)\big]$ is also para-unitary. From the relations $L(z) = A(z) - \mu_1(z) I$, $L(z)v(z) = 0$, $V_1(z)^PV_1(z) = I$ and $L(z)^P=L(z)$ it is clear that
	\[
	V_1(z)^P A(z) V_1(z) - \mu_1(z)I = 
	V_1(z)^P L(z) V_1(z) = \begin{bmatrix}
	0 & \\
	& \wt L_1(z) 
	\end{bmatrix}
	\implies 
	V_1(z)^P A(z) V_1(z) =
	\begin{bmatrix}
	\mu_1(z) & \\
	& L_1(z) 
	\end{bmatrix}
	\]
	where $L_1(z) = \wt L_1(z) + \mu_1(z)$ is still para-Hermitian, and $L_1(z_0)$ has eigenvalues $(\mu_2(z_0),\dots,\mu_n(z_0))$ for any $z_0\in S^1$. We can thus proceed by induction, producing a sequence of para-unitary matrices $V_i(z)$ such that 
	\[
	V_k(z)^P\cdots V_2(z)^PV_1(z)^P A(z) V_1(z)V_2(z)\cdots V_k(z)  =
	\begin{bmatrix}
	\mu_1(z) & & &\\
	&\ddots & &\\
	& & \mu_k(z) &\\
	& & & L_k(z) 
	\end{bmatrix}
	\]
	with $L_k(z)$ para-Hermitian for all $k\le n$. If now $U(z) = \prod_{i=1}^{n} V_i(z)$ we obtain the unitary analytic eigendecomposition $A(z) = U(z)D(z)U(z)^P$ with $[D(z)]_{i,i} = \mu_i(z)$ for all $i=1,\dots,n$. 
\end{proof}

\begin{remark}
The function $z\mapsto \ol z$ is, of course,  not holomorphic on any open domain. However, let $\Omega \subset \C$ be either a circle or a line. Then, it can be shown that there exists a holomorphic M\"obius transformation that coincide with $\ol z$ on $\Omega$. Namely,
\begin{itemize}
\item Given $\beta \in \C$ and $\theta \in ]-\pi,\pi]$, consider the (generic) line $\Omega=\{ z= t e^{\textnormal i \theta} + \beta \ \mathrm{for} \ \mathrm{some} \ t \in \mathbb{R}\}$. On this line, ${\ol z}\mid_\Omega = \left(\alpha z + \gamma\right)\mid_\Omega$ where $\alpha=e^{-2 \textnormal i \theta}$ and $\gamma={\ol \beta}-\beta  e^{-2 \textnormal i \theta}$.
\item Given $\beta \in \C$ and $\rho>0$, consider the (generic) circle $\Omega=\{z=\beta+\rho e^{ \textnormal i a} \ \mathrm{for} \ \mathrm{some} \ a \in ]-\pi,\pi]\}$. On this circle, ${\ol z} \mid_\Omega= \frac{{\ol \beta} z + \alpha}{z-\beta}\mid_\Omega$ where $\alpha=\rho^2-|\beta|^2$.
\end{itemize}
Conversely, if $f(z)$ is a M\"{o}bius transform, then it either has the form $f(z)=\alpha z + \beta$ or the form $f(z)=(\alpha z+\beta)/(z+\delta)$.
In the former case, imposing that the solutions to $f(z)={\ol z}$ are not just isolated point yields $|\alpha|=1$, and we obtain that the locus of the solutions is a generic line. 
In the latter case, again imposing that $f(z)={\ol z}$ has more than just isolated solutions, we get that $\alpha=-{\ol \delta}$, $-|\delta|^2<\beta \in \mathbb{R}$, and the locus of the solutions is a generic circle.
Important special cases include the following: if $\Omega=\mathbb{R}$ then $\ol z \mid_{\f R} = z \mid_{\f R}$; if $\Omega=\textnormal i \mathbb{R}$ then $\ol z \mid_{\textnormal i \f R}= -z\mid_{\textnormal i \f R}$; if $\Omega=S^1$ then $\ol z\mid_{S^1} = z^{-1} \mid_{S^1}$. 
Hence, Rellich's Theorem \ref{thm:rellich} generalizes to the case of matrices that are Hermitian over a generic line, and Theorem \ref{theo:analytic_eigvl_imply_analytic_eigvc} generalizes to the case of matrices that are Hermitian over a generic circle.

In general, let $K$ be a connected set and suppose that $\ol z = g(z)$ on $K$ where $g(z)$ is holomorphic on some open connected neighbourhood of $K$. In addition, it is possible to prove that $g(z)$ is injective (up to restricting the domain). As a consequence, $f(z) \mapsto \ol{f(g^{-1}(\ol z))}$ is an isomorphism of the ring $\mathcal H(K)$ such that $\ol{f(g^{-1}(\ol z))}=\ol{f(z)}$ for any $z\in K$. Theorem \ref{theo:analytic_eigvl_imply_analytic_eigvc} can thus be generalized to matrices with entries in $\mathcal H(K)$ who are Hermitian with respect to the conjugation $f(z) \mapsto \ol{f(g^{-1}(\ol z))}$. 
\end{remark}

\subsection{Existence of the unitary eigendecomposition over $\PC$}\label{sec:existenceinPC}

In Subsection \ref{sec:goodevals}, we have proved that any para-Hermitian matrix $A(z) \in \HCM{n}$ whose eigenvalues are all analytic on $S^1$ has an analytic unitary eigendecomposition over $\mathcal H(S^1)$. In \cite[Theorem 1]{EVD2}, it was claimed that more generally an $n \times n$ para-Hermitian matrix in $\HCM{n}$ has a unitary eigendecomposition as a product of matrices that lie in the space $\PC^{n \times n}$. While the theorem is true, the proof in \cite{EVD2} had a gap that we aim to fill in the present subsection. In fact, we are going to prove a stronger statement. We start by generalizing Definition \ref{def:pHc} to matrices over $\PC$.

\begin{definition}\label{def:pHc2}
Given $A(z)= [a_{i,j}(z)]_{1 \leq i \leq m,1 \leq j \leq n} \in \PC^{m \times n}$, its para-Hermitian conjugate is defined as $A(z)^P = [\ol {a_{j,i}(\ol z^{-1})}]_{1 \leq i \leq m,1 \leq j \leq n}$. Moreover a matrix $R(z) \in \PC^{n \times n}$ is called a para-Hermitian matrix if $R(z)=R(z)^P$ and a matrix $U(z) \in \PC{n \times n}$ is called para-unitary if $U(z)U(z)^P=U(z)^PU(z)=I$.
\end{definition}

We are now going to prove that every para-Hermitian $A(z) \in \mathcal P(S^1)^{n\times n}$ have an analytic unitary eigendecomposition on the space $\mathcal P(S^1)$. We thus will obtain \cite[Theorem 1]{EVD2} (that assumed the stronger requirement $A(z) \in \HCM{n}$) as a corollary. Our starting step is a reordering lemma for the eigenvalues.  Here and below, recall that given a permutation $\sigma \in S_n$ and an element $i \in \{1,\dots,n \}$, the \emph{orbit} of $i$ is the set of points in the cycle containing $i$; the orbits of $\sigma$ are thus the subsets of $\{1,\dots,n\}$ corresponding to each cycle of $\sigma$.

\begin{lemma}\label{lem:permutation}
Let $A(\theta)$ be an analytic $2\pi N$-periodic $n\times n$ Hermitian matrix on $\f R$ with analytic eigenvalues $\mu_1(\theta),\dots,\mu_n(\theta)$. There exists a permutation $\sigma$ on $\{1,\dots,n\}$  satisfying $\mu_i(\theta + 2\pi N) = \mu_{\sigma(i)}(\theta)$ for all $i$ and all $\theta$. One can find such a permutation where if $\mu_i(\theta)$ has period $2\pi kN$ with $k$ the least possible positive integer to obtain a period, then the orbit of $\sigma$ containing $i$ has length $k$.
\end{lemma}
\begin{proof}
	Let $\Gamma:= \{\mu_1(\theta), \dots, \mu_q(\theta)\}$ be,  without loss of generality up to a relabelling of the indices, the subset of distinct eigenvalues. Recall that, here, $\mu_i$ and $\mu_j$ are considered equal if they are the same analytic function of $\theta$; or in other words, if $\mu_i(\theta)=\mu_j(\theta)$ $\forall$ $\theta\in \f R$. %(or equivalently, they coincide on any open non empty subset of $\f R$).
	Since the set of points where any two distinct analytic functions coincide is a discrete set \cite[Theorem 10.18]{Rudin}, there necessarily exists $\theta_0\in ]-\pi,\pi[$ such that  $\mu_1(\theta_0), \dots, \mu_q(\theta_0)$ are $q$ distinct values.
	Since $A(\theta_0)= A(\theta_0 + 2\pi N)$, they have the same eigenvalues up to permutation, so there exists a permutation $\sigma$ such that $\mu_i(\theta_0 + 2\pi N) = \mu_{\sigma(i)}(\theta_0)$. 
	The eigenvalues in $\Gamma$ are distinct in some open neighbourhood $\Omega$ of $\theta_0$, so  $\mu_i(\theta + 2\pi N) = \mu_{\sigma(i)}(\theta)$ for all $i=1,\dots,n$ and for all $\theta \in \Omega$ in such neighbourhood. To conclude the proof, observe that by the identity theorem \cite[Corollary of Theorem 10.18]{Rudin} they must actually coincide on all $\f R$.

	Consider the orbit of the index $i$, that is,  $\{\sigma^{\circ k}(i) = j : k\in \f N\}$.
	Given now an eigenvalue $\mu_i(\theta)$ with period  $2\pi pN$  where $p$ is the least possible positive integer to obtain a period, let $q$ be the length of the orbit of $\sigma$ containing $i$. Necessarily we have $p|q$, and moreover $\mu_i(\theta) = \mu_i(\theta + 2\pi pN) = \mu_{\sigma^{\circ p}(i)} (\theta)$ and this relation holds  for every $\mu_j(\theta)$ with $j$ in the same orbit of $i$. 
	If the orbit is $(i=i_1,i_2,i_3,\dots,i_q)$ then we have $\mu_{i_s}(\theta) = \mu_{i_{r}}(\theta)$ for every $s$ and every $r$ such that $r \equiv s+p \mod q$. This is enough to conclude that we can replace this orbit with $(i_1,\dots i_p)(i_{p+1},\dots, i_{2p})\dots(i_{q-p+1},\dots,i_q)$ and the new permutation still satisfies $\mu_i(\theta + 2\pi N) = \mu_{\sigma(i)}(\theta)$ and additionally the orbit of $\mu_i$ is now of the correct length $p$. 
	Repeating the same argument for all the eigenvalues, we end up with a permutation with the correct orbit lengths.   
\end{proof}

We are now ready to prove the main result in this Subsection, that is, Theorem \ref{theo:EVD_in_not_puiseux} below.

\begin{theorem}\label{theo:EVD_in_not_puiseux}
	Let $A(z) \in \PC^{n\times n}$ be a para-Hermitian matrix. There exists an eigendecomposition $A(z) = U(z)D(z)U(z)^P$, where $D(z) \in \PC^{n\times n}$ is diagonal and real for all $z \in S^1$ and $U(z) \in \PC^{n\times n}$ is para-unitary.
\end{theorem}
\begin{proof}
Note that $\mathcal H_\alpha(S^1) \cu \mathcal H_{\alpha\beta}(S^1)$ for any pair of positive integers $\alpha,\beta$, so we can find an $N$ such that $A(z)\in \mathcal H_{N}(S^1)$, namely if $a_{i,j}(z) \in \mathcal H_{\alpha_{i,j}}(S^1)$, then take $N$ as the least common multiple of all the $\alpha_{i,j}$. 
As a consequence there exists $B(z)\in \HC$ such that $A(z) = B(z^{1/N})$. Moreover $A(z_0)$ is Hermitian for all $z_0\in S^1$ since $A(z)$ is para-Hermitian.  
Let
% $\wt A(\theta) := A(\exp(\textnormal i\theta )) = B(\exp(\textnormal i\theta/N ))$ on $\theta\in ]-\pi,\pi[$ and 
$A(\theta) : =  B(\exp(\textnormal i\theta/N ))$ on $\theta \in \f R$ and notice it coincides with $A(\exp(\textnormal i\theta ))$ for $\theta\in (-\pi,\pi]$  and it is $2\pi N$-periodic. 
Since $A(\theta)$ is analytic and Hermitian on $\f R$, by Rellich's theorem there exists a unitary eigendecomposition $A(\theta) = U(\theta) D(\theta) U(\theta)^*$ with analytic entries on $\f R$. 
Call $\mu_i(\theta)$ the analytic diagonal entries of $D(\theta)$.
%and let $\Gamma:= \{\mu_1(\theta), \dots, \mu_q(\theta)\}$ be,  without loss of generality, the subset of distinct elements, where $\mu_i$ and $\mu_j$ are considered equal if they coincide for all $\theta\in \f R$ (or equivalently, they coincide on an open non empty subset of $\f R$). Since the set of intersection between two different analytic functions is a discrete set, there necessarily exists $\theta_0\in ]-\pi,\pi[$ such that  $\mu_1(\theta_0), \dots, \mu_q(\theta_0)$ are $q$ distinct values.
From Lemma \ref{lem:permutation}, there exists a permutation $\sigma \in S_n$ such that $\mu_i(\theta + 2\pi N) = \mu_{\sigma(i)}(\theta)$, and given its associated permutation matrix $P_\sigma$ we obtain $D(\theta+2\pi N) = P_{\sigma}D(\theta) P_\sigma^T$
%By definition, $A(\theta)$ is $2\pi N$ periodic, so the entries of $D(\theta_0)$ and $D(\theta_0 + 2\pi N)$ coincide for any $p\in \f Z$ up to permutations, since are both the eigenvalues of $A(z_0)$. Notice moreover that since $\mu_1(\theta_0), \dots, \mu_q(\theta_0)$ are distinct, for each $\mu_i$ there must exists a $\mu_j$ with the same multiplicity such that $\mu_i(\theta+2\pi N) = \mu_j(\theta)$ in an open neighbourhood of $\theta_0$, and thus for all $\theta\in \f R$. Since $D(\theta_0)$ and $D(\theta_0 + 2\pi Np)$ have the same diagonal entries up to permutation, there exists a permutation $\sigma$ and its associated permutation matrix $P_\sigma$ such that $D(\theta_0+2\pi N) = P_{\sigma}D(\theta_0) P_\sigma^T$, and again, since they coincide on an open neighbourhood of $\theta_0$, they actually coincide on all $\f R$. 
If now $L$ is the order of $\sigma$, then $P_\sigma^L = I$ and $D(\theta + 2\pi NL) = D(\theta)$. If we call $M:= NL$, this proves that all the eigenvalues $\mu_i(\theta)$ are $2\pi M$-periodic analytic functions and as a consequence $\mu_i(M\theta)$ are $2\pi$-periodic and analytic. The matrix $A(M\theta) = B(\exp(\textnormal iL\theta))$ is $2\pi$-periodic (actually $2\pi/L$-periodic) and analytic, and one can conclude that $B(z^L)\in\HCM n$ has eigenvalues $\lambda_i(z)\in \HC$ with $\lambda_i(\exp(\textnormal i\theta)) =  \mu_i(M\theta)$ for any $\theta$. From Theorem \ref{theo:analytic_eigvl_imply_analytic_eigvc}, $B(z^L)$ admits an eigendecomposition $B(z^L) = V(z) \Sigma(z) V(z)^P$ on $\HC$ where $\Sigma(z)$ is diagonal and $[\Sigma(z)]_{i,i} = \lambda_i(z)$ for all $i=1,\dots,n$. This allows us to conclude that 
\[
A(z) = B(z^{1/N}) = B((z^{1/M})^L) = V(z^{1/M}) \Sigma(z^{1/M}) V(z^{1/M})^P
\]
is the sought eigendecomposition over $\PC$.
\end{proof}

As anticipated, an immediate consequence of Theorem \ref{theo:EVD_in_not_puiseux} is that every para-Hermitian matrix $A(z)\in \HCM n$ admits a $z^{1/N}$-analytic unitary eigendecomposition over $\PC$, or more specifically over $\mathcal H_N(S^1)$ for some integer $N\ge 1$. 

\begin{example}
The para-Hermitian matrix $R(z) \in \HCM n$ of Example \ref{ex1} admits the unitary eigendecomposition
\[
\left( \frac{1}{\sqrt{2}}\begin{bmatrix}
1& z^{1/2}\\
-1 & z^{1/2}
\end{bmatrix}\right)\begin{bmatrix}
z^{1/2}+z^{-1/2} & 0\\
0 & -z^{1/2}-z^{-1/2}
\end{bmatrix}\left( \frac{1}{\sqrt{2}}\begin{bmatrix}
1& z^{1/2}\\
-1 & z^{1/2}
\end{bmatrix}\right)^P  =:U(z) D(z) U(z)^P.\]
It is clear that neither $U(z)$ nor $D(z)$ belong to $\HC^{2 \times 2}$; they both do, however, belong to $\mathcal H_2(S^1)^{2 \times 2}$ and hence, a fortiori, to $\PC^{2 \times 2}$.
\end{example}

Let $A(z) \in \HC^{n \times n}$. Then, $A(z)$ has a unitary eigendecomposition in $\PC^{n\times n}$ by Theorem \ref{theo:EVD_in_not_puiseux}, but this eigendecomposition is in general not unique. However, one can actually prove that the existence of one eigendecomposition for which the periods of eigenvalues and their associated eigenvectors coincide. This is stated more precisely in Proposition \ref{lem:periodicity_coincide} below.

\begin{proposition}\label{lem:periodicity_coincide}
	Let $A(z) \in \HC^{n\times n}$ be a para-Hermitian matrix, and suppose $\mu_1(z),\dots,\mu_n(z)\in \PC$ are the eigenvalues of $A(z)$. Moreover, for all $i=1,\dots,n$, let $\alpha_i \in \mathbb{N}$ be the smallest integer such that $\mu_i(z) \in \mathcal H_{\alpha_i}(S^1)$. Then, there exists an eigendecomposition $A(z) = U(z)D(z)U(z)^P$, where (1) $D(z)$ is diagonal with $[D(z)]_{i,i}=\mu_i(z)$ for all $i=1,\dots,n$ and (2) $U(z)$ is para-unitary and its $i$-th column $u_i(z):=U(z)e_i$ belongs to $\mathcal H_{\alpha_i}(S^1)$ for all $i=1,\dots,n$.
\end{proposition}
\begin{proof}
Let $\alpha := \alpha_1$ and $\mu(z) :=\mu_1(z)\in \mathcal H_\alpha(S^1)$, where $\mu(z)$ has multiplicity $q$ as an eigenvalue of $A(z)$. With the potential exception of some special values of $\theta$, that nevertheless  must belong to a discrete subset of $\f R$, $\mu(\theta) := \mu(\exp(\textnormal i\theta))$ is an eigenvalue of $A(\theta) := A(\exp(\textnormal i\theta))$ of algebraic multiplicity precisely $q$; it follows that, for almost all $\theta$, the relative eigenspace has dimension $q$. Suppose that $A(z) = \wt U(z) D(z)\wt U(z)^P$ is a unitary eigendecomposition whose existence is guaranteed by Theorem \ref{theo:EVD_in_not_puiseux}. Then, let $\wt V(z)$ be the $n\times q$ matrix whose columns are the eigenvectors in $\wt U(z)$ associated to $\mu(z)$, and $\wt V(\theta):=  \wt V(\exp(\textnormal i\theta))$. By the observation above, for almost every $\theta$ the matrix $A_1(\theta) := \wt V(\theta)\mu(\theta)\wt V(\theta)^*$ is a multiple of the projection on the eigenspace relative to $\mu(\theta)$. On the other hand, since $A(\theta)$ is $2\pi$-periodic and $\mu(\theta)$ is $2\pi\alpha$-periodic,  then $A_1(\theta)$ is a $2\pi\alpha$-periodic analytic matrix whose eigenvalues are also $2\pi\alpha$-periodic. From Theorem \ref{theo:analytic_eigvl_imply_analytic_eigvc} applied to $A_1(\alpha\theta)$, we have an analytic $2\pi\alpha$-periodic eigenvalue decomposition $A_1(\theta) = U_1(\theta)D_1(\theta)U_1(\theta)^*$ with just two distinct eigenvalues, that is, $\mu(\theta)$ with multiplicity $q$ and $0$ with multiplicity $n-q$. In particular, one can write $A_1(\theta) = V(\theta)\mu(\theta) V(\theta)^*$  with $V(\theta)$ being a $n\times q$ matrix with orthonormal columns that spans the correct eigenspace and with $2\pi\alpha$-periodic entries. This proves that in the eigendecomposition $A(z) = \wt U(z) D(z)\wt U(z)^P$ we can substitute the eigenvectors relative to $\mu(z)$ with the columns of $V(z)\in \mathcal H_{\alpha}(S^1)^{n\times d}$ associated to $V(\theta)$. Repeating the same procedure for all the eigenvalues, we obtain a unitary analytic eigendecomposition in which eigenvectors and eigenvalues have the same periods.
\end{proof}

\begin{remark}\label{rem:landau}
One may wonder, given the size $n$ of a matrix $A(z) \in \PC^{n \times n}$, how large $N$ could be at worst to guarantee that there exist $U(z),D(z) \in \mathcal H_N(S^1)^{n \times n}$ yielding  $A(z)=U(z)D(z)U(z)^P$.

For all $1\leq i,j \leq n$, define $\alpha_{ij}:=\min_{k \in \mathbb{N}} \{k: a_{i,j}(z) \in \mathcal{H}_k(S^1)\}$. Then, the proof of Theorem \ref{theo:EVD_in_not_puiseux} makes it clear that $N\leq LM$ where $M=\mathrm{lcm}\, \alpha_{i,j}$ and $L$ is the largest possible order of an element of the symmetric group $S_n$. The function $n \mapsto L(n)$ is sometimes referred to as Landau's function (H. Landau studied it in \cite{Landau}). It is possible to prove \cite{Massias} that $L(n) \leq \gamma \exp(\sqrt{n \log n})$ with $\gamma \simeq 1.05313$, and that asymptotically $L(n) \sim \exp(\sqrt{n \log n})$ \cite{Miller}.

In particular, if $A(z) \in \HC^{n \times n}$, then $M=1$ and we have the bound $N\leq L(n) \leq 1.0532 \cdot \exp(\sqrt{n \log n})$. See \cite[page 499]{Miller} for a list of exact values of $L(n)$ for $2 \leq n \leq 19$.

\end{remark}

We conclude this subsection summarizing the observations of Remark \ref{rem:landau} and stating them more formally as a corollary.

\begin{corollary}\label{cor:EVD_in_not_puiseux}
	Let $A(z) \in \HC^{n\times n}$ be a para-Hermitian matrix, and let $L(n)$ be the largest possible order of an element of the symmetric group $S_n$. Then, there exists an integer $N\leq L(n)$ and an eigendecomposition $A(z) = U(z)D(z)U(z)^P$, where $D(z) \in \mathcal{H}_N(S^1)^{n\times n}$ is diagonal and real for all $z \in S^1$ and $U(z) \in \mathcal{H}_N(S^1)^{n\times n}$ is para-unitary.
\end{corollary}

\subsection{Pseudo-circulant decomposition of holomorphic para-Hermitian matrices}

In Subsection \ref{sec:existenceinPC}, we have proved in Theorem \ref{theo:EVD_in_not_puiseux} that a para-Hermitian matrix $A(z)\in \HCM n$ always admits an eigendecomposition over $\PC$ (in fact over $\mathcal{H}_N(S^1)$ for some $N\leq L(n)$, by Corollary \ref{cor:EVD_in_not_puiseux}), and in Proposition \ref{lem:periodicity_coincide} that the eigenvectors can be chosen with the same period of the associated eigenvalues. Moreover, from Lemma \ref{lem:permutation} we know that there exists a permutation $\sigma \in S_n$ that changes the ordering of the eigenvalues order when considering points at a distance of precisely $2\pi$. The orbits of $\sigma$ partition the eigenvalues into $k$ subsets $C_\ell = \{\mu_{\ell,1},\dots, \mu_{\ell,\alpha_\ell}\}$ ($\ell=1,\dots,k$), such that there exists $\lambda_\ell(z)\in \HC$ for which 
\[
\mu_{\ell,\beta}( \exp(\textnormal i\alpha_\ell \theta) ) = \lambda_\ell( \exp(\textnormal i \theta) \exp(\textnormal i \pi (2\beta-1)/\alpha_\ell ) )  ) 
\quad \forall \theta \in (-\pi,\pi]/\alpha_\ell, \,\, \forall \beta=1,\dots,\alpha_\ell, \,\,\forall \ell=1,\dots,k.
\]In particular, for all $\ell,j$, the eigenvalues $\mu_{\ell,j}$ belong to $\mathcal H_{\alpha_\ell}(S^1)$. In what follows we show that $A(z)$ can be block diagonalized by para-unitary matrices over $\HC$, where the blocks correspond to the eigenvalue subsets $C_\ell$.
We report a full proof following the lead of \cite{EVD2}. The first step is to establish two technical lemmata.

\begin{lemma}\label{lem:D_N_e_F_N}
	Given a positive integer $N$, let $F_N$ be the $N\times N$ orthogonal Fourier matrix whose $(i,j)$ element is, for any $i,j=1,\dots,N$, equal to $N^{-1/2}\exp(\textnormal i 2\pi(i-1)(j-1)/N)$. Moreover, let $D_N(\theta)$ be the $N\times N$ diagonal matrix whose $(i,i)$ element is, for any $i=1,\dots,N$, equal to $\exp(\textnormal i \theta(i-1)/N)$. Then 
	\[
	D_N(\theta+2\pi) F_N = 
	D_N(\theta) F_N P_N, 
	\qquad P_N = \begin{bmatrix}
	& & & 1\\
	1 & & & \\
	& \ddots & &\\
	& & 1 &
	\end{bmatrix}.
	\]
\end{lemma}
\begin{proof}
From a direct computation, we see that $D_N(\theta+2 \pi)=D_N(\theta)D_N(2 \pi)$. On the other hand, for any $1\leq i,j \leq N$,
\[ [D(2 \pi) F_N]_{ij} = \frac{1}{\sqrt{N}} e^{\textnormal i 2\pi\frac {i-1}N} \cdot 
e^{\textnormal i 2\pi\frac {(i-1)(j-1)}N}
 = \frac{1}{\sqrt{N}} e^{\textnormal i 2\pi\frac {(i-1)j}N} = [F_N P_N]_{ij}.
  \]
  \end{proof}
  \begin{lemma}\label{lem:pseudo_circulant}
	Let $\mu_1(\theta),\dots, \mu_N(\theta)$ be $2\pi N$-periodic analytic functions on $\f R$ such that 
	\[
	\mu_{k}(\theta) = \mu_1(\theta+ 2\pi(k-1)), \quad \forall k.
	\]
	and let $M(\theta)$ be the $N\times N$ diagonal matrix whose $(i,i)$ element, for all $i=1,\dots,N$, is equal to $\mu_i(\theta)$.
	Given $D_N(\theta)$ and $F_N$ as in Lemma \ref{lem:D_N_e_F_N}, the matrix
	\[
	C(\theta) := D_N(\theta) F_N M(\theta)F_N^* D_N(\theta)^*
	\]
	is analytic, $2\pi$-periodic and pseudo-circulant, i.e., there exist analytic $2\pi$-periodic functions $\phi_0(\theta),\dots, \phi_{N-1}(\theta)$ such that 
	\[
	C(\theta) = 
	\begin{bmatrix}
	\phi_0(\theta) &e^{-\textnormal i\theta}\phi_{N-1}(\theta) &\dots & e^{-\textnormal i\theta}\phi_{1}(\theta)\\
	\phi_1(\theta) & \phi_0(\theta)& \ddots & \vdots\\
	\vdots& \ddots & \ddots& e^{-\textnormal i\theta}\phi_{N-1}(\theta)\\
	\phi_{N-1}(\theta)&\dots &\phi_1(\theta)  &\phi_0(\theta)
	\end{bmatrix}.
	\]
 Further, if $\mu_i(\theta)$ are real-valued, then $C(\theta)$ is also Hermitian. 
\end{lemma}
\begin{proof}
	All matrices involved are analytic, so $C(\theta)$ is also analytic. Moreover, by Lemma \ref{lem:D_N_e_F_N}, 
	\[
	C(\theta+2\pi) = D_N(\theta) F_N
	P_N M(\theta + 2 \pi) P_N^*
	F_N^* D_N(\theta)^*.
	\]
	On the other hand, 
	\[
	P_N M(\theta + 2 \pi) P_N^*
	=
\begin{bmatrix}
\mu_N(\theta+2\pi) & & & \\
& \mu_1(\theta+2\pi) & & \\
& &  \ddots & \\
& & & \mu_{N-1}(\theta+2\pi)
\end{bmatrix}
=
	M(\theta)
	\]
	so we find $C(\theta+2\pi) = C(\theta)$. Furthermore,
	\begin{align*}
			C(\theta)_{i,j} &= 
	\sum_{k=1}^N	
		D_N(\theta)_{i,i} (F_N)_{i,k} \mu_k(\theta)
		(F_N^*)_{k,j} D_N(\theta)^*_{j,j}\\
		&= 
		\frac 1N
		\sum_{k=1}^N	
		e^{\textnormal i \theta\frac {i-1}N} 
		e^{\textnormal i 2\pi\frac {(i-1)(k-1)}N}
		 \mu_k(\theta)
		 e^{-\textnormal i 2\pi\frac {(k-1)(j-1)}N}
		 e^{-\textnormal i \theta\frac {j-1}N}\\
		 &= 
		 \frac 1N
		 e^{\textnormal i \theta\frac {(i-j)}N} 
		 \sum_{k=1}^N	
		 		 \mu_k(\theta)
		 e^{\textnormal i 2\pi\frac {(i-j)(k-1)}N}
	\end{align*} 
is a function of $i-j$. Thus, $C(\theta)$ is a Toeplitz matrix and we can define $C(\theta)_{j+q,j}=:\phi_q(\theta)$ where $q:=i-j$ indexes the diagonals of $C(\theta)$. Showing that $C(\theta)$ is pseudo-circulant is tantamount to proving that $e^{-\textnormal i\theta}\phi_q(\theta)=\phi_{q-N}(\theta)$ for all $0<q \leq N-1$. To this goal, note that for $q>0$ it holds
 	\begin{align*}
 \phi_{q}(\theta)
 &= 
 \frac 1N
 e^{\textnormal i \theta\frac {q}N} 
 \sum_{k=1}^N	
 \mu_k(\theta)
 e^{\textnormal i 2\pi\frac {q(k-1)}N}\\
  &= 
 \frac 1N
 e^{\textnormal i \theta}
 e^{\textnormal i \theta\frac {q-N}N} 
 \sum_{k=1}^N	
 \mu_k(\theta)
 e^{\textnormal i 2\pi\frac {(q-N)(k-1)}N}\\
 &=  e^{\textnormal i \theta}\phi_{q-N}(\theta).
 \end{align*} 
 Finally, it is immediate to check that $\phi_q(\theta)^* = \phi_{-q}(\theta)$ when $\mu_i(\theta)$ are all real, so  in this case $C(\theta)$ is also Hermitian. 
\end{proof}

Now, we can use Lemma \ref{lem:D_N_e_F_N} and Lemma \ref{lem:pseudo_circulant} to show Theorem \ref{theo:pseud-circulant_dec} below.

\begin{theorem}\label{theo:pseud-circulant_dec}
	Let $A(z) \in \HC^{n\times n}$ be a para-Hermitian matrix.
	There exists a decomposition $A(z) = U(z)D(z)U(z)^P$ in $\HC$ where $U(z)$ is para-unitary and $D(z)$ is block diagonal with pseudo-circulant blocks.
\end{theorem}
\begin{proof}
	From Lemma \ref{lem:periodicity_coincide}, we know that there exists a unitary eigendecomposition $A(z) = V(z)\Sigma(z) V(z)^P$ over  $\mathcal P(S^1)$ having the property that the periods of the eigenvalues $\mu_i(z):=e_i^T \Sigma(z) e_i$ and of the respective eigenvectors $v_i(z):=V(z)e_i$ are the same. Moreover, from Lemma \ref{lem:permutation} in the case $N=1$, we know that there exists a permutation $\sigma$ such that $\mu_i(\theta+2\pi) = \mu_{\sigma(i)}(\theta)$ where for all $j$ we set $\mu_j(\theta):=\mu_j(\exp(\textnormal i \theta))$. Note that if $q$ is the multiplicity of the eigenvalue $\mu_i(z)$, then also $\mu_j(z)$ must have multiplicity $q$ for all $j$ in the orbit of $i$. Suppose without loss of generality that $O:=\{\mu_1(z), \dots, \mu_M(z) \}$ is an arbitrary orbit of $\sigma$ with $\mu_i(\theta) = \mu_{1}(\theta + (i-1)2\pi)$ for $1\le i\le M$. Necessarily, for all $i$ such that $\mu_i(z)\in O$, the function $\mu_i(\theta)$ must have period $2\pi M$ and represent an eigenvalue of  multiplicity $q$. Call now $V_i(z)$ the submatrix of $V(z)$ whose columns are $q$ orthonormal eigenvectors associated with the eigenvalue $\mu_i(z)$. By construction, $V_i(z) \in \mathcal H_M(S^1)^{n\times q}$ for $1\le i\le M$, and we can define again $V_i(\theta):=V_i(\exp(i \theta))$ which is, for all $i$, an analytic and $2 \pi M$-periodic function. Take now $\theta_0\in \f R$ with the property that any two distinct eigenvalues are different when evaluated at $\theta=\theta_0$. In particular this implies that the eigenspaces $\mc E_i(\theta_0)$ associated with all distinct $\mu_i(\theta_0)$ are pairwise orthogonal and generated by the columns of $V_i(\theta_0)$. Moreover, since $A(\exp(\textnormal i \theta))$ is $2\pi$-periodic, also at $\theta=\theta_0 +2\pi k$ it must have pairwise orthogonal eigenspaces associated with distinct eigenvalues, and this must hold for all $k\in \f Z$. For $1\le i\le M$, we have $\mu_i(\theta_0) = \mu_{1}(\theta_0 + (i-1)2\pi)$, so $\mc E_i(\theta_0) = \mc E_1(\theta_0 + (i-1)2\pi)$ and the respective eigenprojections are also equal. Thus, $V_i(\theta_0)\mu_i(\theta_0)V_i(\theta_0)^* = V_1(\theta_0 + (i-1)2\pi)\mu_{1}(\theta_0 + (i-1)2\pi)V_1(\theta_0 + (i-1)2\pi)^*$. Since the set of possible values of $\theta_0$ for which this argument is valid is a dense open subset of $\f R$ \cite[Theorem 10.18]{Rudin}, and every function involved is analytic (in $\theta$), by the identity theorem \cite[Corollary of Theorem 10.18]{Rudin}  we get the functional identity
	\[
	V_i(\theta)\mu_i(\theta)V_i(\theta)^* = V_1(\theta + (i-1)2\pi)\mu_{1}(\theta + (i-1)2\pi)V_1(\theta + (i-1)2\pi)^*,\quad  i=1,\dots, M.
	\]
	This shows that we can replace in $V(z)$ the columns corresponding to $V_i(\theta)$ with the columns corresponding to $V_1(\theta + (i-1)2\pi)$, without losing the validity of the unitary eigendecomposition $A(z) = V(z)\Sigma(z)V(z)$ since
	\[
	A(z) = 
	\sum_{i:\mu_i(z)\in O} V_i(\theta)\mu_i(\theta)V_i(\theta)^*
	+
	\sum_{i:\mu_i(z)\not\in O} V_i(\theta)\mu_i(\theta)V_i(\theta)^*.
	\]
If now $v_i(\theta)$ is the leftmost column of $V_i(\theta)$, and thus an eigenvector relative to $\mu_i(\theta)$, we have $v_i(\theta) = v_1(\theta + (i-1)2\pi)$ for $1\le i\le M$ and these vectors are all $2\pi M$-periodic. Call $\wt V(\theta)$ the matrix whose columns are $v_1(\theta),\dots,v_M(\theta)$ and $\wt D(\theta)$ the $M \times M$ diagonal matrix whose $(i,i)$th element is $\mu_i(\theta)$. From Lemma \ref{lem:pseudo_circulant} we obtain that
\[
\wt V(\theta)\wt D(\theta)\wt  V(\theta)^* = 
\wt V(\theta)F_M^* D_M(\theta)^* C(\theta)  D_M(\theta) F_M \wt V(\theta)^* =: W(\theta) C(\theta) W(\theta)^*
\]
where  $W(\theta) = \wt V(\theta)F_M^* D_M(\theta)^*$. On the other hand, Lemma \ref{lem:D_N_e_F_N} implies that 
\begin{align*}
	W(\theta + 2\pi) &= \wt V(\theta+2\pi)P_M^*F_M^* D_M(\theta)^* \\
	&=
	\begin{bmatrix}
	v_2(\theta) & \dots & v_M(\theta) & v_1(\theta)
	\end{bmatrix}P_M^*
	F_M^* D_M(\theta)^* \\
	& =
	\wt V(\theta)F_M^* D_M(\theta)^*
	=
	W(\theta)
\end{align*}
so that all matrices in $W(\theta) C(\theta) W(\theta)^*$ are $2\pi$-periodic, $C(\theta)$ is pseudo-circulant and $W(\theta)$ has still orthonormal columns. This is enough to conclude that we can substitute the eigenvalues in $O$ (without multiplicity) and their associated eigenvectors with $C(\theta)$ and $W(\theta)$. Repeating the reasoning for all orbits, we obtain an orthonormal basis as the union of the columns of all $W(\theta)$  and a block-diagonal matrix with pseudo-circulant block entries $C(\theta)$, and every entry is $2\pi$-periodic.
\end{proof}
%{\color{red}Questa del 3.14 mi ha fatto venire il mal di mare. Ma non trovo facilmente modo di semplificarla... Controlla bene che non abbia introdotto errori nel sistemarla.}

\begin{remark}
	The permutation $\sigma$ that realizes Lemma \ref{lem:permutation} may not be unique when there are eigenvalues with multiplicity greater than one, but from Lemma \ref{lem:permutation} we know that there exists one where the orbit lengths match the periods of the eigenvalues. This permutation allows us to write a decomposition $A(z) = U(z)D(z)U(z)^P$ with the smallest possible pseudo-circulant diagonal blocks. Notice that when all the eigenvalues have period $2\pi$, this decomposition coincides with the EVD of Theorem \ref{theo:analytic_eigvl_imply_analytic_eigvc}.
\end{remark}

\section{Application to the singular value decomposition of matrices over $\PC$}\label{sec:svd}
If $A(z) \in \PC^{m \times n}$ is not necessarily para-Hermitian (and generally not even necessarily square), it is of interest to analyze its functional singular value decomposition. We argue that the latter exists provided that one accepts two relaxation with respect to the usual requirements of a singular value decomposition: the singular values must be allowed to be possibly negative, and are not necessarily ordered.

Theorem \ref{thm:svdinPS1} below makes this claim more precise. Note that it is an extension of related results in \cite{Weissnew} for matrices in $\HC^{m \times n}$; our proof, however, is different and inspired by the ideas in \cite{BunseMehrmann}.
\begin{theorem}\label{thm:svdinPS1}
Let $A(z) \in \PC^{m\times n}$. There exists a singular value decomposition $A(z) = U(z)S(z)V(z)^P$, where $S(z) \in \PC^{m\times n}$ is diagonal and real for all $z \in S^1$ and $U(z) \in \PC^{m\times m}, V(z) \in \PC^{n \times n}$ are both para-unitary.
\end{theorem}
\begin{proof}
Consider the para-Hermitian matrix
\[ H(z) = \begin{bmatrix}
0 & A(z)\\
A(z)^P & 0
\end{bmatrix} \in \PC^{(m+n) \times (m+n)}; \]
by Theorem \ref{theo:EVD_in_not_puiseux}, there is a unitary eigendecomposition over $\PC$ of the form $H(z)=Q(z) D(z) Q(z)^P$. Moreover, it can be verified that $\lambda(z)$ is a nonzero eigenvalue of the matrix $H(z)$ with eigenvector $\begin{bmatrix}
b(z)\\
c(z)
\end{bmatrix}$ and $b(z)\in\mathcal P(S^1)^m$, $c(z)\in\mathcal P(S^1)^n$  if and only if $-\lambda(z)$ is a nonzero eigenvalue of $H(z)$ with eigenvector $\begin{bmatrix}
b(z)\\
-c(z)
\end{bmatrix}$. In particular, by taking these eigenvectors to be orthonormal, we can pick selected columns of $Q(z)$ and selected rows and columns of $D(z)$ and write the equation
\[ H(z)=\begin{bmatrix}
B(z) & B(z)\\
C(z) & -C(z)
\end{bmatrix} \begin{bmatrix}
\Lambda(z) & 0\\
0 & -\Lambda(z)
\end{bmatrix}\begin{bmatrix}
B(z) & B(z)\\
C(z) & -C(z)
\end{bmatrix}^P,\]
where $\begin{bmatrix}
B(z) & B(z)\\
C(z) & -C(z)
\end{bmatrix} \in \PC^{(m+n) \times 2r}$ has orthonormal columns and $\Lambda(z) \in \PC^{r \times r}$ is nonsingular. The previous equation implies in turn that $A(z)C(z)=B(z)\Lambda(z)$ and that $A(z)^P B(z)=C(z) \Lambda (z)$. Hence, noting that $\sqrt{2} B(z)$ and $\sqrt{2} C(z)$ also have orthonormal columns by construction, we have that $A(z)=[\sqrt{2} B(z)] \Lambda(z) [\sqrt{2} C(z)]^P$ is a ``compact" singular value decomposition. To complete the proof, note that $\sqrt{2} B(z)$ can be completed to a square invertible matrix $M(z)$: this follows by observing that $\mathcal H_N(S^1)$ is a PID for any $N$ and again by \cite[Theorem 3.3]{zaballa}. We can then apply Lemma \ref{lem:GS} to obtain a para-unitary matrix $U(z)$ whose leftmost columns coincide with those of $\sqrt{2} B(z)$. A similar procedure can be employed to construct $V(z)$ from $\sqrt{2} C(z)$; finally, we can define 
\[ S(z) = \begin{bmatrix}
\Lambda(z) & 0_{r \times (n-r)}\\
0_{(m-r) \times r} & 0_{(m-r) \times (n-r)}
\end{bmatrix}\]
and find that $A(z) = U(z)S(z)V(z)^P$.
\end{proof}

\begin{example}
Consider the matrix \[ A(z)=\begin{bmatrix}
1 & z
\end{bmatrix} \in \HC^{1 \times 2} \Rightarrow A(z)^P = \begin{bmatrix}
1\\
z^{-1}
\end{bmatrix} \in \HC^{2 \times 1}.\]
Proceeding as in the proof of Theorem \ref{thm:svdinPS1} we find 
\[  H(z) = \begin{bmatrix}
0 & A(z)\\
A(z)^P & 0
\end{bmatrix} = \begin{bmatrix}
z/\sqrt{2} & z/\sqrt{2}\\
z/2 & -z/2\\
1/2 & -1/2
\end{bmatrix}\begin{bmatrix}
\sqrt{2} & 0\\
0 & -\sqrt{2}
\end{bmatrix}\begin{bmatrix}
z/\sqrt{2} & z/\sqrt{2}\\
z/2 & -z/2\\
1/2 & -1/2
\end{bmatrix}^P \]
and thus $B(z)=z$, $C(z) = \frac{1}{\sqrt{2}} \begin{bmatrix}
z\\
1
\end{bmatrix} \Rightarrow V(z) = \frac{1}{\sqrt{2}}\begin{bmatrix}
z & -z\\
1 & 1
\end{bmatrix}$. Hence,
\[ A(z)=\underbrace{z}_{U(z)} \cdot \underbrace{\begin{bmatrix}
\sqrt{2} & 0
\end{bmatrix}}_{S(z)} \cdot \underbrace{\left( \frac{1}{\sqrt{2}}\begin{bmatrix}
z^{-1} & 1\\
-z^{-1} & 1
\end{bmatrix}\right)}_{V(z)^P}   \]
is the sought singular value decomposition, which in this particular case exists over $\HC$.\end{example}
\begin{example}
For an example of a matrix over $\HC$ whose singular value decomposition only exists over $\PC$, but not over $\HC$, consider the scalar $A(z)=1+z \in \HC$. Since $|1+z| \not \in \HC$ there is no hope to decompose $A(z)$ with an $S^1$-analytic singular value decomposition. However, 
\[ A(z) = \underbrace{z^{1/2}}_{U(z)} \cdot  \underbrace{(z^{1/2}+z^{-1/2})}_{S(z)} \cdot \underbrace{1}_{V(z)^P} \]
is a singular value decomposition over $\mathcal{H}_2(S^1) \subset \PC$.
\end{example}

\section{Application to the sign characteristics of $*$-palindromic matrix polynomials}\label{sec:signchar}

\subsection{Sign characteristics of Hermitian matrix-valued functions}
Let $H(z)$ be a matrix-valued function which is analytic for all $z$ in an open domain $\Omega \subseteq \C$. In numerical linear algebra, the \emph{finite eigenvalues} of such a matrix-valued function $H(z)$ are defined \cite[Section 2]{mntx} as the numbers $z_0 \in \Omega$ such that $\mathrm{rank} H(z_0) < \sup_{z \in \Omega} \mathrm{rank} H(z)$. In particular, if $H(z)$ is square and regular, i.e., if its determinant $\det H(z)$ is not the zero function, the finite eigenvalues are the roots of $\det H(z)$. Since the set of  analytic functions on a domain is, from the algebraic perspective, an EDD \cite{EDD}, one defines the \emph{partial multiplicities} of a finite eigenvalue $z_0$ of $H(z)$ as the multiplicities of $z_0$ as a root of the not identically zero diagonal elements in the Smith canonical form of $H(z)$. In particular, if the only nonzero partial multiplicity is equal to $1$, a finite eigenvalue is said to be \emph{simple}. 

Let us now further assume that $\Omega$ contains the real line and that $H(z)$ is a Hermitian matrix-valued function, that is, $H(x)=H(x)^*$ for all $x \in \f R$. (Note that this implies that Theorem \ref{thm:rellich} applies to $H(z)$.) Observe that the finite eigenvalues of $H(z)$, defined above, are different than the eigenvalues of the matrix $H(z)$ in the functional sense of Theorem \ref{thm:rellich} and of the previous Sections in this paper -- in particular, the finite eigenvalues of $H(z)$ are complex numbers while the eigenvalues of the matrix $H(z)$ are real-analytic functions of the variable $z$. However, it is clear by Theorem \ref{thm:rellich} itself that there is a connection: since a matrix which is unitary on the real line must be regular, the finite eigenvalues must be precisely the roots of the nonzero eigenvalue functions.
Now suppose that $H(z)=U(z)D(z)U(z)^*$ is a Rellich unitary eigendecomposition as in Theorem \ref{thm:rellich}, and let $\lambda \in \mathbb{R}$ be a real root of a not identically zero diagonal element $D_{i,i}(z)$ of $D(z)$, i.e., a real finite eigenvalue of $H(z)$. Following \cite[Definition 2.3]{mntx}, if $D_{i,i}(z)=\epsilon_i c_i (z-\lambda)^{m_i} + O(z-\lambda)^{m_i+1}$ where $\epsilon_i \in \{-1,1\}$ and $c_i > 0$, then one says that $m_i$ is the $i$-th partial multiplicity of the finite eigenvalue $\lambda$ (and it can be proved that this definition agrees with the one given above based on the Smith canonical form), and $\epsilon_i$ is its $i$-th sign characteristics. Moreover, in this case, the $i$-th sign feature of $\lambda$ is defined as $\epsilon_i (1-(-1)^{m_i})/2$ (hence, a sign feature can be $-1$, $0$, or $+1$). Furthermore, if $H(z)$ is polynomial, one can also define partial multiplicities, sign characteristics and sign features associated with $\infty$, see \cite[Definition 2.8]{mntx}. It is discussed in \cite{mntx} that the sign characteristics and features are important for the perturbation analysis of the real eigenvalues; in particular a pair of nearby real eigenvalues can be removed by the real line by a small perturbation only if the sum of their sign features is $0$ \cite[Theorem 5.13]{mntx}. There are also alternative, equivalent, algebraic definitions of the sign characteristics of a Hermitian matrix polynomial that can be shown to be equivalent to the analytic definition mentioned above: see \cite{glr,lr,mntx} and the references therein. Here, it suffices to say that if $\lambda$ is a simple real eigenvalue of a regular Hermitian matrix polynomial $P(z)$, associated with right (and left) eigenvector $v$, then the sign characteristic of $\lambda$ is $v^* P'(\lambda) v$ where $P'(z)$ denotes the derivative of $P(z)$ with respect to $z$.

\subsection{Sign characteristics of $*$-palindromic matrix polynomials}

Given some matrices $P_i \in \C^{n \times n}$, $i=0,\dots,g$, let $P(z)=\sum_{i=0}^g P_i z^i$ be a matrix polynomial. Here $g$ is an integer, called the \emph{grade of palindromicity}, greater than or equal to the degree of $P(z)$. In other words, $g$ is bounded below by the degree of the matrix polynomial $P(z)$, but we admit the possibility that $P_g=0$ and hence $\deg P(z) < g$. 

We say that $P(z)$ is $*$-palindromic \cite{m4}  if $P_i=P_{g-i}^*$ for all $0 \leq i \leq g$. When $g$ is even, it is not difficult to verify that $P(z)$ is $*$-palindromic of degree $g$ if and only if $R(z):=z^{-g/2} P(z) \in \HC^{n \times n}$ is a para-Hermitian matrix whose elements are functions analytic on the unit circle $S^1$. When $g$ is odd, then $z^{-g/2}$ has a branch point at $z=-1$ and thus $R(z)$ is not holomorphic on the whole $S^1$; however, in this case $R(z) \in \mathcal{H}_2(S^1)^{n \times n} \subset \PC^{n \times n}$, so that Theorem \ref{theo:EVD_in_not_puiseux} still applies. Moreover, $R(z)$ retains the property of being Hermitian on the unit circle and it is analytic on $S^1 \setminus \{-1\}$. In addition, it is clear that $z \in S^1 \setminus \{-1\}$ is a finite eigenvalue of $P(z)$ if and only if it is a finite eigenvalue of $R(z)$. The analysis of this paper suggests one analytic definition of the sign characteristics (and features) of a $*$-palindromic matrix polynomial; it is expected that these objects also play a crucial role in the perturbation theory for this class of matrix polynomials, analogously to what happens for the Hermitian case.  Indeed, this can be proved using the result for Hermitian matrices and the definition that we will give for the sign characteristic of $*$-palindromic matrix polynomials, and we will illustrate the relevance of the result for structured perturbation theory with some examples. 

Let us first start from the case of even  $g$. We know that for some $N$ the Rellich eigenvalue functions of $R(z)$ are analytic functions of $w=z^{1/N}$. In other words, they are expressible as analytic and $2 \pi N$-periodic functions $F_i(\theta)$, where $z=\exp(\textnormal i\theta)$ (see Subsection \ref{sec:Laurent}, Theorem \ref{theo:EVD_in_not_puiseux}, and related remarks); these functions are in particular analytic in $]-\pi,\pi[$. Hence, if $-1 \neq \lambda \in S^1$ is a unimodular eigenvalue of $P(z)$, we may consider the Taylor expansions around $\theta_0$ of those nonzero eigenvalue functions $F_i(\theta)$ that have a zero at $\theta_0$, with $\lambda=\exp(i \theta_0)$, say,
\[ F_i(\theta)=\epsilon_i c_i (\theta-\theta_0)^{m_i} + O(\theta-\theta_0)^{m_i+1},\]
where $\epsilon_i$ is a sign and $c_i > 0$. Analogously to the Hermitian case, we can then define $\epsilon_i$ to be the $i$-th sign characteristic, $m_i$ to be the $i$-th partial multiplicity, and the $i$-th sign feature can be defined accordingly, in the same way as in the case of Hermitian matrix polynomials.
\begin{example}
Consider $P(z)=\begin{bmatrix}
2z&1+z\\
z^2+z&2z
\end{bmatrix}=z R(z)$ where $R(z)$ is the para-Hermitian matrix introduced in \cite{EVD2} . The finite eigenvalues of $P(z)$ are the roots of $\det P(z)=-z^3+2z^2-z$. Hence, $z_0=1$ is a unimodular eigenvalue of algebraic multiplicity $2$. The eigenvalue functions of $R(z)$ are $F_1(\theta)=2+2\cos(\theta/2)$ and $F_2(\theta)=2-2\cos(\theta/2)$. Only the latter has a root at $\theta_0=0$, corresponding to $z_0=\exp(\textnormal i \theta_0)=1$. We have the Taylor expansion
\[ F_2(\theta)= (+1) \frac{1}{4} \theta^2 + o(\theta^2), \]
and hence the eigenvalue $1$ of $P(z)$ has its  only nonzero partial multiplicity equal to $2$, sign characteristic $+1$, and sign feature $0$.
\end{example}

Note that the $F_i(\theta)$ are the eigenvalues of the Hermitian matrix function \begin{equation}\label{eq:H}
    H(\theta):=R(\exp (\textnormal i \theta))=\exp (-\frac{\textnormal i g \theta}{2})P(\exp (\textnormal i \theta)).
\end{equation}  Suppose that $P(z)$, and hence $H(\theta)$, is regular. Observe that $z_0=\exp(\textnormal i \theta_0) \in S^1 \setminus \{-1\}$ is a simple unimodular eigenvalue of $P(z)$ if and only if $\theta_0 \in ]-\pi,\pi[$ is a simple real eigenvalue of $H(\theta)$. Let $v$ be a corresponding right (and left) eigenvector; then, taking into account that $P(e^{i \theta_0})v=0$, in this case the sign characteristic can also be computed as the sign of
\begin{equation}\label{eq:signformula}
 \left[ v^* \frac{d H(\theta)}{d \theta} v\right]_{\theta=\theta_0} = \textnormal i \frac{z_0}{z_0^{g/2}} \left[ v^* \frac{d P(z)}{d z}  v\right]_{z=e^{\textnormal i \theta_0}}.  
\end{equation} 
We thus recover the characterization of \cite[Proposition 12.6.1]{glrind}

For the case of odd grade of palindromicity $g$, $R(z)$ as defined above is generally not analytic on the whole $S^1$. Nevertheless, $H(\theta)$ as in \eqref{eq:H} is still analytic and $4 \pi$-periodic (see Subsection \ref{sec:Laurent}); thus, its eigenvalues are still analytic on $]-\pi,\pi[$ by Rellich's Theorem \ref{thm:rellich},  and there they coincide with the eigenvalues of $R(z)$. Hence, the approach proposed above is still sensible; the choice of the branch of $z^{g/2}$ in \eqref{eq:signformula} should be coherent with that in the equation $P(z)=z^{g/2} R(z)$ when constructing $H(\theta)$. We can thus extend the above definition and results to odd grade polynomials; note that in \cite{glrind} only the regular and even grade (in fact also even degree, as the leading coefficient was taken to be invertible) case was treated.

For ease of reference, we formally collect the outcome of the previous analysis in the following Definition \ref{def:sign}, in which the above mentioned care about the uniformity of the choice of a branch of the square root function must be taken.

\begin{definition}\label{def:sign}
Let $P(z)$ be an $n \times n$ $*$-palindromic matrix polynomial with grade of palindromicity $g$. Suppose that $-1 \ne \lambda \in S^1$ is a unimodular finite eigenvalue of $P(z)$. Define $R(z):=z^{-g/2}P(z) \in \mathcal{H}_2(S^1)^{n \times n}$, and let $R(z)=U(z) D(z) U(z)^P$ its unitary eigendecomposition as in Theorem \ref{theo:EVD_in_not_puiseux}. Furthermore, let $N \in \f N$ be such that $U(z),D(z) \in \mathcal{H}_N(S^1)^{n \times n}$ and, for all $i=1,\dots,n$ let $F_i(\theta):=[D(w(\theta))]_{i,i}$ where $w=z^{1/N}$ and $w(\theta)=\exp(\textnormal i \theta/N)$. Then, there is at least a value of $i$ for which $F_i(\theta)$ is not the zero function, it has a zero $\theta_0 \in ]-\pi,\pi[$ with $\lambda=\exp(i \theta_0)$, and it admits the corresponding Taylor expansion
\[ F_i(\theta)=\epsilon_i c_i (\theta-\theta_0)^{m_i} + O(\theta-\theta_0)^{m_i+1},\]
where $\epsilon_i \in \{-1,+1\}$, $0<m_i \in \f N$ and $c_i > 0$. Moreover, assume without loss of generality that the $F_i(\theta)$ are ordered in such a way that (1) if $j>i$ and $F_i(\theta) \equiv0$, then $F_j(\theta)\equiv 0$ (2) if $j>i$ and $F_i(\theta) \not \equiv 0$ has a zero of order $m_i$ at $\theta=\theta_0$, then either $F_j(\theta)\equiv 0$ or $F_j(\theta) \not \equiv 0$ has a zero of order $m_j \geq m_i$ at $\theta=\theta_0$. In this setting, we say that:
\begin{enumerate}
    \item if $F_i(\theta) \not \equiv 0$, then the corresponding $m_i$ is the $i$-th partial multiplicity of the eigenvalue $\lambda$;
    \item if $m_i >0$ then $\epsilon_i$ is the $i$-th sign characteristic of the eigenvalue $\lambda$;
    \item if $m_i >0$ then 
    \[ \phi_i = \epsilon_i \frac{1-(-1)^{m_i}}{2}\]
    is the $i$-th sign feature of the eigenvalue $\lambda$.
\end{enumerate}
In the special case where $P(z)$ is regular and $\lambda$ is simple, let $H(\theta)=R(\exp(\textnormal i \theta))$ and suppose that $\lambda$ is associated with the right eigenvector $v$. Then, the sign characteristic of $\lambda$ is 
\begin{equation}\label{eq:sign2}
   \mathrm{sign} \left[ v^* \frac{d H(\theta)}{d \theta} v\right]_{\theta=\theta_0} = \mathrm{sign} \left(\textnormal i \frac{z_0}{z_0^{g/2}} \left[ v^* \frac{d P(z)}{d z}  v\right]_{z=e^{i \theta_0}} \right).
\end{equation}

\end{definition}

\begin{example}
Consider the $*$-palindromic pencil \[ P(z)=\begin{bmatrix}
2z+2 & z+1-\textnormal i\\
\textnormal iz+z+1 & \textnormal i-\textnormal iz
\end{bmatrix} \Rightarrow P'(z)=\begin{bmatrix}
2 & 1\\
\textnormal i+1 & -\textnormal i
\end{bmatrix}. \]
It can be shown that its two finite eigenvalues are both unimodular and simple. A numerical computation yields the two eigenpairs $\lambda_1 \simeq -0.9582+0.1716\textnormal i$, $v_1^*\simeq \begin{bmatrix}
0.9331-0.0669 \textnormal i & 0.3888+0.0127 \textnormal i
\end{bmatrix}$ and $\lambda_2 \simeq 0.6852+0.7284 \textnormal i$, $v_2^*=\begin{bmatrix}
0.3875+0.2298 \textnormal i & -0.9685+0.0315 \textnormal i
\end{bmatrix}$. We can compute
\[ \textnormal i \lambda_1^{1/2} v_1^* P'(z) v_1 \simeq -2.4454, \qquad \textnormal  i \lambda_2^{1/2} v_2^* P'(z) v_2 \simeq 1.4238;\]
hence, the eigenvalue $\lambda_1$ has sign characteristic (and sign feature) $-1$ while $\lambda_2$ has sign characteristic (and sign feature) $+1$.
\end{example}

\begin{remark}
When $g$ is odd, changing the choice of a branch of the function $z^{g/2}$ in the definition of $R(z)$ within Definition \ref{def:sign}, or in \eqref{eq:sign2} for the case of simple eigenvalues of regular $*$-palindromic polynomials, induces in turn a change of the sign characteristics of the eigenvalue of interest. However, this is a \emph{global} change in the sense that, as long as the same choice of a branch is coherently made for all unimodular eigenvalues, then all the sign characteristics are simultaneously flipped. In practical applications of the sign characteristics, such as signature constraint theorems or perturbation theory results \cite{mntx}, all that matters is whether two (or more) eigenvalues have the same sign characteristic (feature) or different sign characteristics (feature); see \cite{glrind,mntx} for more details. Thus, a global change is unimportant and does not hinder the coherence of the theory. Nevertheless, and for the same reasons outlined above, it is important to make the same choice for all eigenvalues.
\end{remark}

In the next two examples, we illustrate the relation between sign characteristics and structured perturbation theory in the simple case of two nearby simple eigenvalues. In \cite[Section 5]{mntx}, it was shown that if a Hermitian matrix-valued function has two nearby simple real eigenvalues with opposite sign characteristics, then there exist some small Hermitian perturbation that can remove the eigenvalues from the real line; in other words, while of course every slightly perturbed Hermitian function will still have two eigenvalues in a \emph{complex} neighbourhood of the original ones, for some structured perturbations they may have become nonreal. On the contrary, if two nearby simple real eigenvalues have the same sign characteristics, then every slightly perturbed Hermitian matrix-valued function will still have two real eigenvalues in a \emph{real} neighbourhood of the original ones. Definition \ref{def:sign} makes it clear that an analogous scenario happens when considering structured perturbations of $*$-palindromic matrix polynomials. Example \ref{ex:uno} and Example \ref{ex:due} exhibit, respectively, a $*$-palindromic matrix polynomial whose two nearby unimodular eigenvalues have opposite sign characteristics (and can thus be removed from the unit circle by a small structured perturbation) and a $*$-palindromic matrix polynomial whose two nearby unimodular eigenvalues have identical sign characteristics (and thus \emph{cannot} be removed from the unit circle by a small structured perturbation).

\begin{example}\label{ex:uno}
Let $\epsilon > 0$ be a small real positive parameter and define
\[ A_\epsilon=\begin{bmatrix}
1 & \textnormal i\\
\textnormal i & \epsilon^2
\end{bmatrix}, \qquad P_\epsilon(z)=A_\epsilon z + A_\epsilon^* = \begin{bmatrix}
z+1 & \textnormal i (z-1)\\
\textnormal i(z-1) & \epsilon^2 (z+1)
\end{bmatrix}.\]
The matrix polynomial $P_\epsilon(z)$ is $*$-palindromic, and it has two simple unimodular eigenvalues $\lambda_1=\frac{(1+\textnormal i \epsilon)^2}{1+ \epsilon^2}$ and $\lambda_2=\ol{\lambda_1}$. Moreover, associated eigenvectors are $v_1=\begin{bmatrix}
\epsilon & 1
\end{bmatrix}^T$ and $v_2=\begin{bmatrix}
-\epsilon & 1
\end{bmatrix}^T$.
We have (choosing the principle branch of the square root)
\[ \textnormal i \lambda_1^{1/2} v_1^* A_\epsilon v_1 = -2 \epsilon \sqrt{1+\epsilon^2}, \quad \textnormal  i \lambda_2^{1/2} v_2^* A_\epsilon v_2 = 2 \epsilon \sqrt{1+\epsilon^2}. \]
Hence, the two nearby eigenvalues have opposite sign characteristics and thus the theory of sign characteristics allows us to predict that there is a small structured perturbation of $P_\epsilon(z)$ that has no unimodular eigenvalues. To see this, consider the pertubation
\[ \Delta A_\epsilon=\begin{bmatrix}
0&0\\
0&-2 \epsilon^2
\end{bmatrix}, \qquad P_\epsilon(z) + \Delta P_\epsilon(z)=(A_\epsilon + \Delta A_\epsilon) z + (A_\epsilon + \Delta A_\epsilon)^* = \begin{bmatrix}
z+1 &\textnormal i(z-1)\\
\textnormal i(z-1) & -\epsilon^2 (z+1)
\end{bmatrix}.\]
This is a small structured perturbation. Indeed, the coefficients $A_\epsilon$ and $A_\epsilon^*$ in $P_\epsilon(z)$ both have norm $\|A_\epsilon\|_2=\|A_\epsilon^*\|_2=\frac{\sqrt{5}+1}{2}+O(\epsilon^2)=O(1)$; and we have perturbed them by perturbations $\Delta A_\epsilon$ and $\Delta A_\epsilon^*$ of norm $\|\Delta A_\epsilon\|_2 = \| \Delta A_\epsilon^*\|_2 = 2 \epsilon^2$, while still preserving the $*$-palindromic structure of $P_\epsilon(z)$. On the other hand, the finite eigenvalues of $P_\epsilon(z)+\Delta P_\epsilon(z)$ are $\mu_1=\frac{1+\epsilon}{1-\epsilon}$ and $\mu_2=\mu_1^{-1}$, neither of which lies on $S^1$.
\end{example}

\begin{example}\label{ex:due}
Let $\epsilon>0$ be a small real positive parameter, and define
\[ B_\epsilon = \begin{bmatrix}
\textnormal i & \epsilon\\
\epsilon & \textnormal i
\end{bmatrix}, \qquad Q_\epsilon(z)=B_\epsilon z + B^*_\epsilon = \begin{bmatrix}
\textnormal i(z-1) & \epsilon(z+1)\\
\epsilon(z+1) & \textnormal i(z-1)
\end{bmatrix}. \]
The matrix polynomial $Q_\epsilon(z)$ is $*$-palindromic and it has two simple unimodular eigenvalues $\lambda_1$ and $\lambda_2$ (defined in the same way as in Example \ref{ex:uno}). Associated eigenvectors are instead $w_1=\begin{bmatrix}
1 &1
\end{bmatrix}^T$ and $w_2=\begin{bmatrix}
-1 &1
\end{bmatrix}^T$. In this case, still choosing the principal branch of the square root, it holds
\[\textnormal  i \lambda_1^{1/2} w_1^* B_\epsilon w_1 = -2 =  \textnormal i \lambda_2^{1/2} w_2^* B_\epsilon w_2. \]
Thus, the two nearby eigenvalues have the same sign characteristic. Hence, for every sufficiently small perturbation $\Delta B_\epsilon \in \C^{2 \times 2}$, we predict using the theory of sign characteristics that $Q_\epsilon(z)+\Delta Q_\epsilon(z):=(B_\epsilon+\Delta B_\epsilon)z+(B_\epsilon+\Delta B_\epsilon)^*$ will still have two eigenvalues on $S^1$. To verify this prediction, note that\footnote{It is not important to exclude $-1$ in this step, because the eigenvalues of $Q_\epsilon(z)$ are approximately (up to an $O(\epsilon)$ distance) equal to $1$, and hence small enough perturbations of it will also have eigenvalues close to $1$.} $z \in S^1 \setminus \{-1\} \Leftrightarrow w=\frac{1-z}{\textnormal i(1+z)} \in \mathbb{R}$. Thus, $Q_\epsilon(z)+\Delta Q_\epsilon(z)$ has two finite eigenvalues on $S^1$ if and only if $T(w,\Delta B_\epsilon):=[\textnormal i(B_\epsilon^*-B_\epsilon)+X]w+ (B_\epsilon^* + B_\epsilon + Y)$ has two real finite eigenvalues, where $X:=\textnormal i(\Delta B_\epsilon^*-\Delta B_\epsilon)$ and $Y:=\Delta B_\epsilon^*+\Delta B_\epsilon$ are both Hermitian matrices of sufficiently small norm. Suppose now $\|\Delta B_\epsilon\|_2 < 1$, implying $\|X\|_2 < 2$. Observe that $\textnormal i(B_\epsilon^* - B_\epsilon)=2I_2$ is positive definite, and thus so is the leading coefficient of $T(w,\Delta B_\epsilon)$ (under the assumption $\| \Delta B_\epsilon\|_2<1$); and it is a classical result \cite{hmt,Parlett} that the finite eigenvalues of a Hermitian pencil $w H_1 + H_0$ whose leading coefficient $H_1$ is positive definite are real (and coincide with the eigenvalues of the Hermitian matrix $-H_1^{-1/2} H_0 H_1^{-1/2}$).
\end{example}

We conclude this section by noting that the above described approach leaves it open how to give an analytic definition of the sign characteristic at the unimodular eigenvalue $-1$ for $*$-palindromic matrix polynomials. We note that one can follow (at least) two possible strategies:
\begin{itemize}
\item Either one could employ an ad hoc approach, just as it was done in \cite{mntx} at $\infty$ for the Hermitian case (with the disadvantage of treating the point $-1$ specially);
\item Or,  since there must exist a $\tau \in \  ]-\pi,\pi]$ such that $\exp(\textnormal i \tau)$ is not an eigenvalue, one could make the non-standard choice of placing the branch line of the logarithm on the semiline $\arg(\theta)=\tau$ (with the disadvantage that such an approach requires a polynomial-dependent definition).
\end{itemize} 
A full treatment of the sign characteristic of $*$-palindromic matrix polynomials, including algebraic formulae to compute the sign characteristics for nonsimple unimodular eigenvalues, or for unimodular eigenvalues of nonregular $*$-palindromic matrix polynomials, is nevertheless beyond the scope of the present article and left as future research.

\section{Conclusions}\label{sec:conclusions}
In this paper we have revisited the existence of unitary eigendecompositions of a para-Hermitian matrix $A(z)$. We have filled some gaps in the existing literature on this subject, delivering the fist (to our knowledge) fully complete proof of this result. Moreover, we have relaxed the assumptions on the para-Hermitian matrix $A(z)$ allowing it to be an analytic (on $S^1$) function of $w=z^{\frac1M}$ for some positive integer $M$, we have clarified that the periods of eigenvalues and eigenvectors may be taken equal in such unitary eigendecompositions, we have explained that similar generalizations of Rellich's theorem can be given for matrices that are analytic and Hermitian on an \emph{arbitrary} circle or line in the complex plane, and we have also given a complete proof of the existence of a holomorphic pseudo-circulant decomposition of a holomorphic para-Hermitian matrix. Furthermore, we have applied our results to the singular value decomposition of matrices that are analytic functions of $w$ on the unit circle, and to the sign characteristics of $*$-palindromic matrix polynomials.
\section*{Acknowledgements}

We thank Jen Pestana, Ian Proudler and Stephan Weiss for introducing us to the importance of this problem in signal processing, for sharing a preliminary version of \cite{Weissnew}, and for useful discussions.

\end{document}